\documentclass[10pt,reqno]{amsart}
\usepackage{amsmath,amscd,amsfonts,amsthm,amssymb,latexsym,mathrsfs}
\usepackage{hyperref}
\hypersetup{
    bookmarks=true,         % show bookmarks bar?
    unicode=false,          % non-Latin characters in AcrobatÕs bookmarks
    pdftoolbar=true,        % show AcrobatÕs toolbar?
    pdfmenubar=true,        % show AcrobatÕs menu?
    pdffitwindow=false,     % window fit to page when opened
    pdfstartview={FitH},    % fits the width of the page to the window
    pdftitle={My title},    % title
    pdfauthor={Author},     % author
    pdfsubject={Subject},   % subject of the document
    pdfcreator={Creator},   % creator of the document
    pdfproducer={Producer}, % producer of the document
    pdfkeywords={keyword1} {key2} {key3}, % list of keywords
    pdfnewwindow=true,      % links in new window
    colorlinks=true,       % false: boxed links; true: colored links
    linkcolor=blue,          % color of internal links
    citecolor=blue,        % color of links to bibliography
    filecolor=blue,      % color of file links
    urlcolor=gray           % color of external links
}
\usepackage{tikz}
\usetikzlibrary{matrix,arrows,patterns}
\usepackage{url}

\theoremstyle{plain}
   \newtheorem{theorem}{Theorem}[section]
   
   \newtheorem{lemma}[theorem]{Lemma}
   \newtheorem{corollary}[theorem]{Corollary}
   \newtheorem{problem}{Problem}

\theoremstyle{definition}

\theoremstyle{remark}
   \newtheorem{remark}[theorem]{Remark}

\def\kkk{\kern.2ex\mbox{\raise.5ex\hbox{{\rule{.35em}{.12ex}}}}\kern.2ex}
\newcommand{\LP}{\mathscr{L{\kkk}P}}
\newcommand {\bC} {\mathbb {C}}
\newcommand{\PP}{\mathscr{P}}
\newcommand{\LL}{\mathscr{L}}
\newcommand{\HH}{\mathscr{H}}
\newcommand {\bR} {\mathbb {R}}
\newcommand {\bN} {\mathbb {N}}

\renewcommand{\Im}{{\rm Im}}

\title[Multiplier Sequences for Laguerre Polynomials]
{A characterization of multiplier sequences for generalized Laguerre bases}

\author[P.~Br\"and\'en]{Petter Br\"and\'en}
\address{Department of Mathematics, Royal Institute of Technology, SE-100 44 Stockholm,
Sweden}
\email{pbranden@kth.se}
\thanks{The first author is a Royal Swedish Academy of Sciences Research Fellow
  supported by a grant from the Knut and Alice Wallenberg
  Foundation. The research is also supported by the G\"oran Gustafsson Foundation.}
\author[E.~Ottergren]{Elin Ottergren}
\address{Department of Mathematics,
            Stockholm University,
            S-10691, Stockholm, Sweden}
\email{ottergren@math.su.se}



\begin{document}
\thispagestyle{empty}

\numberwithin{equation}{section}

\begin{abstract}
We give a complete characterization of  multiplier sequences for generalized Laguerre bases. We also apply our methods to give a short proof of  the characterization of Hermite multiplier sequences achieved by Piotrowski. 
\end{abstract}

\maketitle

\section{Introduction}
In this paper we study linear operators on real polynomials that preserve the property of having only real zeros (we consider constant  polynomials as being real--rooted).  
P\'olya and Schur characterized such linear operators that act diagonally with respect to the standard basis of $\mathbb{R}[x]$, see \cite{PS}. A complete characterization of linear operators preserving real--rootedness was achieved only recently by  Borcea and the first author in \cite{BB1}. However, generalizations of the P\'olya--Schur theorem of the following form are still open in many important cases:

\begin{problem}[Problem 4.2 in \cite{BC}]\label{pr1}
Let $\PP= \{P_n(x)\}_{n=0}^\infty$ be a basis for $\bR[x]$.
For a sequence $\{\lambda_n\}_{n=0}^{\infty}$ of real numbers, define a linear  operator
$T:\mathbb{R}[x] \rightarrow \mathbb{R}[x]$ by $$T(P_n(x)) = \lambda_n
P_n(x), \quad \mbox{ for all } n \in \bN:=\{0,1,2,\ldots\}.$$ Characterize the sequences  $\{\lambda_n\}_{n=0}^\infty$ for which $T$ preserves real--rootedness. 
\end{problem}
We call such a sequence  a $\PP$-\emph{multiplier sequence}, while the term 
\emph{multiplier sequence} is reserved for the classical case $\PP=\{x^n\}_{n=0}^\infty$. 
The case of Problem \ref{pr1} when $\PP=\{x^n\}_{n=0}^\infty$ goes back to Laguerre and Jensen and was completely solved by P\'olya and Schur in \cite{PS}, see also \cite{CC1,L}. Tur\'an \cite{Tu} and subsequently Bleecker and Csordas \cite{BC} provided classes of multiplier sequences for the Hermite polynomials $\HH= \{H_n(x)\}_{n=0}^\infty$, while Piotrowski completely characterized $\HH$-multiplier sequences in \cite{P}. Recently partial results regarding multiplier sequences for the generalized Laguerre bases \cite{FP}, and for the Legendre bases \cite{BDFU}, were achieved.  

Recall that the \emph{(generalized) Laguerre  polynomials}, $\LL_\alpha= \{L_n^{(\alpha)}(x)\}_{n=0}^{\infty}$, are defined by  
\begin{equation}\label{lagdef}
L_n^{(\alpha)}(x)=\sum_{k=0}^{n}{n+\alpha \choose n-k}\frac{(-x)^k}{k!}, \quad \alpha > -1,  
\end{equation}
see \cite[p. 201]{Ra}. 
In this paper we give a complete characterization of $\LL_\alpha$-multiplier sequences for each $\alpha >-1$. We say that a  sequence $\{\lambda_n\}_{n=0}^\infty$ is trivial if there is a number $k \in \bN$ such that $\lambda_n=0$ for all $n \not \in \{k,k+1\}$. It is not hard to see that all trivial sequences are $\LL_\alpha$-multiplier sequences, see \cite[Proposition 2.1]{FP}. Hence it remains to characterize non-trivial $\LL_\alpha$-multiplier sequences, which is achieved by the following: 

\begin{theorem}\label{main}
Let $p(y)=\sum_{k=0}^{\infty}{k+\alpha \choose k} a_ky^k$ be a formal power series, where $\alpha > -1$, and let $\{\lambda_n\}_{n=0}^{\infty}$ be a non-trivial sequence defined by 
$$\lambda_n := \sum_{k=0}^{n}a_k{n \choose k}.$$
Then $\{\lambda_n\}_{n=0}^{\infty}$ is an $\LL_\alpha$-multiplier sequence if and only if $p(y)$ is a real-rooted polynomial with all its zeros contained in the interval $[-1,0].$
\end{theorem}

\begin{remark}
Note that Theorem \ref{main} implies that each non-trivial $\LL_{\alpha}$-multiplier sequence is a polynomial in $n$, i.e., there is a polynomial $P(x)$ such that $\lambda_n=P(n)$ for all $n \in \bN$. Hence the corresponding operator $T$ is a finite order differential operator. 
\end{remark}

\begin{remark}
We may express an arbitrary sequence $\{\lambda_n\}_{n=0}^{\infty}$ as
$\lambda_n = \sum_{k=0}^{n} a_k {n \choose k}$, where the sequence $\{a_k\}_{k=0}^\infty$ is defined by $a_k = \sum_{j=0}^k(-1)^{k-j}{k \choose j} \lambda_j$, for each $k \in \bN$. 
It follows (by elementary binomial identities) that the series  $p(y)$, defined in Theorem \ref{main}, may be expressed  in terms of the sequence $\{\lambda_n\}_{n=0}^{\infty}$ as the formal power series
\begin{equation}\label{strs}
p(y) = \frac 1 {(1+y)^{\alpha+1}}\sum_{n=0}^{\infty} \lambda_n {n+\alpha \choose n} \left(\frac {y}{1+y}\right)^n. 
\end{equation}
Hence $\{\lambda_n\}_{n=0}^{\infty}$  is a non-trivial $\LL_\alpha$-multiplier sequence if and only if \eqref{strs} is a real-rooted polynomial with all its zeros contained in the interval $[-1,0].$
\end{remark}

Our method of proving Theorem \ref{main} is applicable to other bases, and in Section~\ref{herms} we give a short proof of the characterization of Hermite-multiplier sequences  due to Piotrowski \cite[Theorem 152]{P}.

\section{Proof of Theorem \ref{main}}
The main tool used to prove Theorem \ref{main} is the characterization of linear preservers of real--rootedness given in \cite{BB1}, which we now describe. 
The \emph{Laguerre--P\'olya class}, $\LP_1(\bR)$,   consists of all real entire functions that are limits, uniformly on compact subsets of $\bC$, of real--rooted polynomials. Laguerre and P\'olya proved that an entire function $\Phi$ is in the Laguerre--P\'olya class if and only it may be expressed in the form 
\begin{equation}\Phi(x) = cx^ne^{\alpha x-\beta x^2}\prod_{k=1}^{\omega}\bigg(1+\frac{x}{x_k}\bigg)e^{-{x}/{x_k}}\label{LP}, \quad \omega \in \mathbb{N}\cup \{\infty\}, \end{equation} 
where $c,\alpha, x_k \in \mathbb{R}$ for all $k$, $\beta \geq 0$,  $n$ is a non-negative integer and 
$\sum_{k=1}^{\infty}x_k^{-2} < \infty.$ A multivariate polynomial $P\in \bC[x_1,\ldots, x_n]$ is called \emph{stable} if $P(x_1,\ldots, x_n)  \neq 0$ whenever $\Im(x_j)>0$ for all $1 \leq j \leq n$. Hence a real univariate polynomial is stable if and only if it is real--rooted. The \emph{Laguerre--P\'olya class} of real entire functions in $n$ variables, $\LP_n(\bR)$,   consists of all real entire functions that are limits, uniformly on compact subsets of $\bC^n$, of real stable polynomials. 

The \emph{symbol} of a linear operator $T: \mathbb{R}[x] \rightarrow \mathbb{R}[x]$  is the formal power series defined  by
$$G_T(x,y):= \sum_{n=0}^{\infty}\frac{(-1)^n T(x^{n})}{n!}y^n.$$

\begin{theorem}[Theorem 5 in \cite{BB1}]\label{BB}
A linear operator $T: \mathbb{R}[x] \rightarrow \mathbb{R}[x]$ preserves real--rootedness if and only if
 \begin{itemize}
\item[(1)] The rank of $T$ is at most two and $T$ is of the form 
$$T(P) = \alpha(P)Q + \beta(P)R,$$
where $\alpha,\beta: \mathbb{R}[x] \rightarrow \mathbb{R}$ are linear functionals and $Q+iR$ is a stable polynomial, or;
\item[(2)] $G_T(x,y)\in \LP_2(\bR)$, or;
\item[(3)] $G_T(-x,y)\in \LP_2(\bR)$. 
\end{itemize}
\end{theorem}
Theorem \ref{BB} suggests that we should find necessary and sufficient conditions for the symbol, $G_T(x,y)$,  of the operator given by $T(L_n^{(\alpha)}(x))= \lambda_nL_n^{(\alpha)}(x)$ to be in  $\LP_2(\bR)$. 
By giving an explicit expression for $G_T(x,y)$, the following Lemma is a step towards establishing such conditions. Though the result follows from 
\cite[Proposition 4.2]{FP}, for the sake of completeness we include a short proof here based on a well known identity for generalized Laguerre polynomials.

\begin{lemma}\label{symbexp1} 
Let $\{\lambda_n\}_{n=0}^\infty$ be a sequence of real numbers. The symbol of the operator $T: \mathbb{R}[x] \rightarrow \mathbb{R}[x]$ defined by $T(L_n^{(\alpha)}(x))=\lambda_nL_n^{(\alpha)}(x)$, for all $n \in \bN$, is given by
$$G_T(x,y)=  e^{-xy}\sum_{n=0}^{\infty}a_ny^nL_n^{(\alpha)}(xy+x),$$
where $$a_n = \sum_{k=0}^{n}(-1)^{n-k}{n \choose k}  \lambda_k, \quad n \in \bN.$$
\end{lemma}

\begin{proof} Recall that the generalized  Laguerre polynomials satisfy the following differential equation (see \cite[p. 204]{Ra}): 
$$nL_n^{(\alpha)}(x) = (x-\alpha-1)\frac{d}{dx}L_n^{(\alpha)}(x)-x\frac{d^2}{dx^2}L_n^{(\alpha)}(x).$$
Consider the operator $\delta :=(x-\alpha-1)d/dx-xd^2/dx^2$
and let 
$$S_k := \frac{\delta(\delta-1)\cdots(\delta-k+1)}{k!}.$$
Then  
$S_kL_n^{(\alpha)}(x) = {n \choose k}L_n^{(\alpha)}(x),$ 
and letting $T$ be the operator corresponding to $\{\lambda_n\}_{n=0}^\infty$, we have 
$T = \sum_{k=0}^{\infty}a_k S_k.$ 
By the change of variables $y= t/(1-t)$ in the generating function for the generalized Laguerre polynomials:
$$
\frac   { e^{-xt/(1-t)} }   {(1-t)^{1+\alpha}}=\sum_{n=0}^{\infty}L_n^{(\alpha)}(x)t^n,
$$
see \cite[p. 202]{Ra}, yields
$$G_{S_k}(x,y)=S_k(e^{-xy}) = \sum_{n=0}^{\infty}{n \choose k} L_n^{(\alpha)}(x)y^n(1+y)^{-n-\alpha-1} .$$
On the other hand, with the same change of variables as above, identity $(9)$ on page 211 in \cite{Ra} states that 
$$\sum_{n=0}^{\infty}{n \choose k} L_{n}^{(\alpha)}(x)y^n(1+y)^{-n-1-\alpha}=  e^{-xy}y^kL_k^{(\alpha)}(xy+x).$$
Hence 
$$
G_T(x,y)= T(e^{-xy})= \sum_{k=0}^\infty a_kS_k(e^{-xy})= e^{-xy}\sum_{k=0}^\infty a_k y^kL_k^{(\alpha)}(xy+x).$$
\end{proof}
The explicit expression \eqref{lagdef} of the Laguerre polynomials now yields:
\begin{equation}
G_T(x,y) = e^{-xy}\sum_{n=0}^{\infty} a_ny^n \sum_{k=0}^{n} {n +\alpha \choose n-k}\frac{(-x(y+1))^k}{k!}.   \label{symbol1} \end{equation}
Setting $p(y)=\sum_{n=0}^{\infty}{n+\alpha \choose n} a_ny^n$ gives 
$$ \frac {p^{(k)}(y)}{(\alpha+1)\cdots (\alpha +k)}= \sum_{n=k}^{\infty}{n+\alpha \choose n-k}a_ny^{n-k}, $$
which together with changing the order of summation in \eqref{symbol1} yields the following consequence of Lemma \ref{symbexp1}:

\begin{corollary}\label{symbexp2}
The symbol of the operator $T: \mathbb{R}[x] \rightarrow \mathbb{R}[x]$ defined by $T(L_n^{(\alpha)}(x))=\lambda_nL_n^{(\alpha)}(x)$, for all $n \in \bN$,  is given by
$$G_T (x,y) = e^{-xy}\sum_{k=0}^{\infty} p^{(k)}(y)\frac{(-xy(y+1))^k}{(\alpha +1)\cdots (\alpha +k)k!},$$
where $p(y)$ is defined as in Theorem \ref{main}. 
\end{corollary}

Before we proceed with the proof of Theorem \ref{main} let us collect some fundamental properties of multiplier sequences in a lemma for ease of reference: 
\begin{lemma}\label{fundament} ${ }$

\begin{enumerate}
\item (P\'olya and Schur, \cite{PS}). Let $\{\lambda_n\}_{n=0}^\infty$ be a sequence of real numbers, and define a formal power series by
$$
\Phi(x):=\sum_{k=0}^{\infty}\lambda_k \frac {x^k} {k!}.
$$
Then $\{\lambda_n\}_{n=0}^\infty$ is a multiplier sequence if and only if $\Phi(x)$ or $\Phi(-x)$ is an entire function that has  the form 
\begin{equation}\label{psrep}
 cx^ne^{sx}\prod_{k=1}^{\omega}\bigg(1+\frac x {x_k}\bigg), \quad  \omega \in \mathbb{N}\cup \{\infty \}, 
\end{equation}
where $s \geq 0$, $n \in \bN$,  $c \in \bR$, $x_k > 0$ for all $k$, and
$\sum_{k=0}^{\omega} x_k^{-1} < \infty$.
\item If  $\{\lambda_n\}_{n=0}^\infty$ is a multiplier sequence and $\lambda_k\lambda_\ell \neq 0$ for some $k<\ell$, then 
$\lambda_i \neq 0$, for all $k \leq i \leq \ell$.
\item  Let $\{\lambda_n\}_{n=0}^\infty$ be a sequence of real numbers and let $T$ be the corresponding diagonal operator. Then $\{\lambda_n\}_{n=0}^\infty$ is a trivial sequence (as defined in the introduction)  if and only if $T$ has rank at most two and $\{\lambda_n\}_{n=0}^\infty$ is a multiplier sequence. 
\end{enumerate}
\end{lemma}
Note that (3) follows easily from (2) which follows easily from (1). 

\subsection{Proof of Necessity} 

Before we proceed with the proof we  recall a version of the  Hermite--Biehler theorem, and the notions of \textit{interlacing zeros} and \textit{proper position}. Let $x_1 \leq x_2 \leq \dots \leq x_n$ and $y_1 \leq y_2 \leq \dots \leq y_m$ be the zeros of two real--rooted polynomials $f$ and $g$, where $\deg f = n,$ $\deg g = m$ and $|n-m| \leq 1$.  We say that the zeros of $f$ and $g$ interlace if they can be ordered so that $x_1 \leq y_1 \leq x_2 \leq y_2 \leq \cdots $ or  $y_1 \leq x_1 \leq y_2 \leq x_2 \leq \cdots$. By convention we also say that the ``zeros'' of any constant polynomial interlace the zeros of any polynomial of degree at most one, and that the zeros of the identically zero polynomial interlace the zeros of any real-rooted (or constant) polynomial. 
 If the zeros of two polynomials $f$ and $g$ interlace, then the \textit{Wronskian}
$$W[f,g] := f'g-fg'$$
is either non-negative or non-positive on the whole of $\mathbb{R}$. 

We say that an ordered pair $f,g$ of real--rooted polynomials are in \emph{proper position}, written $f \ll g$, if the zeros of $f$ and $g$ interlace and $W[f,g]\leq 0$ on $\bR$.

\begin{theorem}[Hermite--Biehler, see e.g. p. 197 in \cite{RS}] \label{HB}
Let $f, g \in \mathbb{R}[x]$, not both identically zero. Then  $f \ll g$ if and only if the polynomial
$g+if$
is stable.
\end{theorem}
We may extend the the notion of proper position to $\LP_1(\bR)$ by setting $f \ll g$ if and only if $g+if \in \LP_1(\bC)$, where $\LP_1(\bC)$ is the \emph{complex Laguerre--P\'olya class} which is defined to be the set of entire functions that are limits, uniformly on compact subsets of $\mathbb{C}$, of stable polynomials in $\bC[x]$. In particular if $g+if \in \LP_1(\bC)$, where $f$ and $g$ are real entire functions, then $W[f,g](x) \leq 0$ for all $x \in \bR$.

Any $\LL_\alpha$-multiplier sequence is a multiplier sequence, see \cite[Lemma 157]{P}. 
Assume that $\{\lambda_n\}_{n=0}^\infty$ is a non-trivial $\LL_\alpha$-multiplier sequence, and let $T$ be the corresponding operator. Then, by Theorem \ref{BB},  $G_T(x,y) \in \LP_2(\bR)$ or $G_T(-x,y) \in \LP_2(\bR)$, since if $T$ has rank at most two then $\{\lambda_n\}_{n=0}^\infty$ is trivial by Lemma \ref{fundament} (3). Assume  $G_T(x,y) \in \LP_2(\bR)$ and expand its expression given  in Corollary \ref{symbexp2} in powers of $x$:
$$
G_T(x,y)= p(y)- x\left(yp(y)+\frac{y(y+1)}{1+\alpha}p'(y)\right)+\cdots.
$$

Non-negative multiplier sequences may be extended to act on functions of two variables by the rule $x^ky^\ell \mapsto \lambda_k x^ky^\ell$ for all $k, \ell \in \bN$. The class $\LP_2(\bR)$ is preserved under this action (see \cite[Lemma 3.7]{BB} and \cite{Br1}). Hence we may truncate the expression above by the multiplier sequence $\{1,0,0,...\}$ and obtain $p(y) \in \LP_1(\bR)$. 

If we instead truncate by the multiplier sequence $\{1,1,0,0,\dots\}$ we arrive at the bivariate expression
\begin{equation} Q(x,y) = p(y)-xq(y)= p(y)-x\Bigg(yp(y)+\frac{y(1+y)}{(1+\alpha)}p'(y)\Bigg) \label{expQ} \end{equation}
which belongs to $\LP_2(\bR)$.  Hence $iQ(i,y) = q(y)+ip(y) \in  \LP_1(\bC)$, and thus 
$$
W[p,q](y)  = -p(y)^2+\frac{y(y+1)}{1+\alpha}(p'(y)^2-p(y)p''(y))-\frac{2y+1}{1+\alpha}p(y)p'(y) \leq 0, 
$$
for all $y \in \bR$. We use the above inequality to prove that all zeros of $p$ are located in the interval $[-1,0]$.

First suppose that $y_0$ is a real and simple zero of $p(y)$. Evaluating $W[p,q]$ at  $y_0$ yields
$y_0(y_0+1)p'(y_0)^2 \leq 0$, 
which can only happen if $y_0\in [-1,0]$.

For multiple zeros we proceed as follows. Consider again $W[p,q](y)$ and a real zero, $y_0$, of $p(y)$ of multiplicity $M \geq 2$.  If $y_0 = 0$ or $y_0=-1$ there is nothing to prove, so assume otherwise. The multiplicity of $y_0$ will be 
$2M$ for $p^2,$ $2M-1$ for $pp'$ and $2M-2$ for $(p')^2$ and $pp''.$ If there is no cancellation the dominating term near $y_0$ of $W[p,q](y)$ is 
\begin{equation}\label{dom}
\frac{y(y+1)}{\alpha+1}(p'(y)^2-p(y)p''(y)).
\end{equation}
However, writing $p(y) = (y-y_0)^Ms(y)$ we obtain 
$$p'(y)^2-p(y)p''(y)= (y-y_0)^{2M-2}\big(Ms(y)^2+(s'(y)^2-s(y)s''(y))(y-y_0)^2 \big),$$
which proves that \eqref{dom} is the dominating term near $y_0$ and from which it follows that $y_0 \in [-1,0]$.

 We know that $p(y)$ is an entire function in $\LP_1(\bR)$ so it has the form \eqref{LP}, and we now show  that it is in fact a polynomial. 
Since its zeros lie in the interval $[-1,0]$, it can only have a finite number of zeros, that is, $p(y)= e^{ay-by^2}K(y)$, where $K(y)$ is a real--rooted polynomial with zeros only in $[-1,0]$, and $a, b \in \bR$ with $b \geq 0$. 
Recall the definition of $Q(x,y)$ in \eqref{expQ}. Then 
$Q(x,y) = e^{ay-by^2}(K(y) - xF(y))$, where 
$$F(y)=yK(y) + \frac{y(y+1)}{1+\alpha}\big((a-2by)K(y)+K'(y)\big).$$
The zeros of $F(y)$ and $K(y)$ interlace by Theorem \ref{HB} (set $x=i$). 
Notice that $\deg F \geq \deg K + 2$, unless $a=b=0$.  Hence $a=b=0$ and there is no exponential factor. 
This completes the proof that $p(y)$ is  a real--rooted polynomial with all its zeros contained in $[-1,0]$, and finishes the proof of necessity in the case when $G_T(x,y) \in \LP_2(\bR)$. It remains to prove that we cannot have $G_T(-x,y) \in \LP_2(\bR)$. 

Assume  $G_T(-x,y) \in \LP_2(\bR)$. Then proceeding as in the case when $G_T(x,y) \in \LP_2(\bR)$, we get $q(y) \ll p(y)$ , where $q(y)$ is as in \eqref{expQ}. Thus 
\begin{equation}\label{minus}
W[p,q](y)  = -p^2(y)+\frac{y(y+1)}{1+\alpha}(p'(y)^2-p(y)p''(y))-\frac{2y+1}{1+\alpha}p(y)p'(y) \geq 0, 
\end{equation}
for all $y \in \bR$. If $p(-1/2)\neq 0$, then Laguerre's inequality (see e.g. \cite[Corollary 3.7]{CC1}) implies that the middle term in \eqref{minus} is non-positive and thus $W[p,q](-1/2) <0$. Suppose $y=-1/2$ is a zero of $p(y)$ of multiplicity $M \geq 1$. Then, since 
$$
(y+1/2)  \frac {p'(y)}{p(y)} \approx M,  
 $$
near $y=-1/2$ we see that also the last term in \eqref{minus} is negative near $y=-1/2$. Hence we cannot have $G_T(-x,y) \in \LP_2(\bR)$.

\subsection{Proof of Sufficiency}
We now prove that the conditions on $p(y)$ in Theorem~\ref{main} imply  $G_T(x,y) \in \LP_2(\bR)$, which will then prove sufficiency by Theorem~\ref{BB}.  Assume that the zeros of 
$$p(y) = \sum_{k=0}^n{k+\alpha \choose k} a_ky^k = \prod_{j=1}^n(y+\theta_j)$$ are real and lie in $[-1,0]$, and consider again the symbol expressed as in Corollary~\ref{symbexp2}:
$$G_T(x,y)=e^{-xy}\sum_{k=0}^{n} p^{(k)}(y)\frac{(-xy(y+1))^k}{(\alpha +1)\cdots (\alpha +k)k!}.$$
Since $\{((\alpha+1)\cdots(\alpha+k))^{-1}\}_{k=0}^{\infty}$ is a non-negative multiplier sequence as proved already by Laguerre \cite{La}, and as such preserves $ \LP_2(\bR)$ when acting on $x$,  it is enough to prove  that
$$\sum_{k=0}^{n}p^{(k)}(y)\frac{(-xy(y+1))^k}{k!} $$
is a stable polynomial in two variables.  Now 
\begin{align*} \sum_{k=0}^{n}p^{(k)}(y)\frac{(-xy(y+1))^k}{k!} &= p(y-xy(y+1))\\
&= \prod_{j=1}^{n}(\theta_j+y-xy(y+1)), \end{align*}
where  $0 \leq \theta_j \leq 1$. Observe that  $-y(y+1) \ll y+\theta$ for all $0 \leq \theta \leq 1$ and thus, by e.g. \cite[Lemma 2.8]{BB2}, it follows that each factor is stable. This finishes the proof of Theorem \ref{main}. \hfill \ensuremath{\Box}

\section{Hermite multiplier sequences}\label{herms}

We will now apply our methods to give a short proof of the characterization of Hermite multiplier sequences achieved by Piotrowski \cite{P}. The Hermite polynomials, $\HH=\{H_n(x)\}_{n=0}^\infty$, may be defined by the generating function 
\begin{equation}\label{genher}
\exp(2xt-t^2) =
  \sum_{n=0}^{\infty}\frac{H_n(x)}{n!}t^n,
 \end{equation} 
 see \cite{Ra}. Since Hermite polynomials are even or odd it is easy to see that $\{\lambda_n\}_{n=0}^\infty$ is an $\HH$-multiplier sequence if and only if $\{(-1)^n\lambda_n\}_{n=0}^\infty$ is an $\HH$-multiplier sequence.   It is also plain to see that any trivial sequence is an  $\HH$-multiplier sequence, and that all $\HH$-multiplier sequences are multiplier sequences (see \cite[Theorem 158]{P}). Since the entries of  multiplier sequences either have the same sign or alternate in sign (by Lemma \ref{fundament} (1)) it remains to characterize non-negative and non-trivial Hermite multiplier sequences. In \cite{P} a generalization of P\'olya's curve theorem led to the following characterization, which we will now re-prove:

\begin{theorem}[Piotrowski, \cite{P}]\label{piot}
Let $\{\lambda_n\}_{n=0}^\infty$ be a non-trivial sequence of non-negative numbers. Then $\{\lambda_n\}_{n=0}^\infty$ is a Hermite multiplier sequence if and only if it is a (classical) multiplier sequence with $\lambda_n \leq \lambda_{n+1}$ for all $n\geq 0.$
\end{theorem}
Let $\{\lambda_n\}_{n=0}^\infty$ be a non-trivial and non-negative classical  multiplier sequence and let $T$ be the corresponding operator. 
Note that \eqref{genher} implies 
$$e^{-xy}= e^{{y^2}/{4}}\sum_{k=0}^{\infty}\frac{H_k(x)}{k!}\bigg(\frac{-y}{2}\bigg)^k,$$ 
and thus the symbol of $T$ has the form
$$G_T(x,y) = T(e^{-xy}) =  e^{{y^2/}{4}}\sum_{k=0}^{\infty}\frac{\lambda_k H_k(x)(-y)^k}{2^k k!}.$$ 
  Since $T$ is of rank greater than two, by Theorem \ref{BB} it remains to classify those non-negative sequences $\{\lambda_n\}_{n=0}^\infty$ for which 
$G_T(x,y) \in \LP_2(\bR)$ or $G_T(-x,y) \in \LP_2(\bR)$. First let us prove that $G_T(-x,y)$ is never in $\LP_2(\bR)$. Suppose that $G_T(-x,y) \in \LP_2(\bR)$ and let $M$ be the first index for which $\lambda_M \neq 0$. Then, since $e^{-y^2/4} \in \LP_2(\bR)$, 
\begin{equation}\label{expgt}
y^{-M}e^{-y^2/4}G_T(-x,y) = \frac{\lambda_M H_M(x)}{2^M M!}+ 
\frac{\lambda_{M+1} H_{M+1}(x)}{2^{M+1} (M+1)!}y + \cdots 
\in \LP_2(\bR). 
\end{equation}
Since  $\{\lambda_n\}_{n=0}^\infty$ is nonnegative, Lemma \ref{fundament} (2) implies $\lambda_M, \lambda_{M+1} >0$, and as in the previous section we may apply the multiplier sequence $\{1,1,0,\ldots\}$ (acting on $y$) to \eqref{expgt}  and conclude  
$$
\frac{\lambda_M H_M(x)}{2^M M!}+ 
\frac{\lambda_{M+1} H_{M+1}(x)}{2^{M+1} (M+1)!}y 
\in \LP_2(\bR). 
$$
Setting $y=i$ yields $H_{M+1}(x) \ll H_M(x)$ which is false (although $H_M(x) \ll H_{M+1}(x)$ is a standard fact about orthogonal polynomials). 

It remains to find necessary and sufficient conditions for $G_T(x,y)$ to belong to the Laguerre--P\'olya class. Note that 
\begin{equation}\label{h1}
\sum_{k=0}^{\infty}\frac{H_k(x)(-y)^k}{2^k k!}= e^{-xy}e^{-{y^2}/{4}} \in \LP_2(\bR).
\end{equation}
Hence for a non-negative  multiplier sequence $\{\lambda_n\}_{n=0}^{\infty}$,  
$$
\sum_{k=0}^{\infty}\frac{\lambda_kH_k(x)(-y)^k}{2^k k!} \in  \LP_2(\bR).
$$
Recall the representation \eqref{LP} of entire functions in $\LP_1(\bR)$. A similar representation holds for functions in $\LP_2(\bR)$  (see \cite[p. 370]{L}):

\begin{theorem}\label{levin}
Let $\mathbb{K}=\bR$ or $\bC$.  If $ f(x,y)$ is an entire function of two variables, then $f$ is in $\LP_2(\mathbb{K})$ if and only if $f$ has the  representation
$$f(x,y)=e^{-ax^2-by^2}f_1(x,y),$$
where $a$ and $b$ are non-negative numbers, and $f_1$ belongs to $\LP_2(\mathbb{K})$ and has genus at most one in each of its variables under the condition that the other variable is fixed in the open upper half-plane. 
\end{theorem}
Thus we may write
\begin{equation}\label{h2}
\sum_{k=0}^{\infty}\frac{\lambda_kH_k(x)(-y)^k}{2^k k!} = e^{-ax^2-by^2}g(x,y)
\end{equation}
for some entire function $g(x,y) \in \LP_2(\bR)$ of genus at most $1$ in each variable under the condition that the other variable is fixed in the open upper half-plane.   Hence
\begin{equation}\label{Hsymb}
G_T(x,y) = e^{{y^2}/{4}}\sum_{k=0}^{\infty}\frac{\lambda_k H_k(x)(-y)^k}{2^k k!} = e^{-ax^2-(b-{1}/{4})y^2}g(x,y).  
\end{equation}
In light of Theorem \ref{levin} our task has reduced to establishing when $b \geq 1/4$.

To this end, recall that the \emph{order} $\rho$, and \emph{type} $\sigma$ of an entire function $f(x)$ may be defined as: 
$$
\rho := \limsup_{r \to \infty}\frac{\log \log M(r)}{\log r} \quad \mbox{ and } \quad \sigma:=\limsup_{r \to \infty}\frac{\log M(r)}{r^{\rho}}, 
$$
where $M(r) := \max_{|z|=r}|f(z)|$. In the definition of the type it is required that the order is finite and non-zero. In terms of its Taylor coefficients, $\{c_n\}_{n=0}^\infty$, the order and type of $f$ are given by 
\begin{equation}\label{ordertype}
\rho = \limsup_{n \to \infty}\frac{n \log n}{\log \frac{1}{|c_n|}} \quad \mbox{ and } \quad (\sigma e \rho)^{{1}/{\rho}} = \limsup_{n \to \infty}n^{{1}/{\rho}}|c_n|^{{1}/{n}},
\end{equation}
see e.g.  \cite[p. 4]{L}. If $c_n=0$, then the quotient is understood to be zero. 

\begin{lemma}\label{ordertypelemma}
 Let $\{\lambda_n\}_{n=0}^\infty$ be a non-negative multiplier sequence with exponential generating function expressed in the form \eqref{psrep}, and  let $g(x)=\sum_{n=0}^\infty c_n x^n$ be an entire function in $\LP_1(\bC)$ of order $2$ and  type $c$.  Then 
\begin{equation}\label{lac}
T(g)(x)=\sum_{n=0}^\infty \lambda_n c_n x^n = \exp(-cs^2x^2)f(x),
\end{equation}
where $f(x)\in \LP_1(\bC)$ has genus at most one. 
\end{lemma}

\begin{proof}
Note that $T(g) \in \LP_1(\bC)$, so it may represented as $T(g)= \exp(-ax^2)f(x)$, where $f(x)\in \LP_1(\bC)$ has genus at most one (by Theorem \ref{levin}). It remains to prove that $a=cs^2$.

Let $\omega$ and $\tau$ be the order and type of \eqref{psrep}, respectively, and let $\rho$ and $\sigma$ be the order and type of $T(g)(x)$, respectively. Then 
$$
\rho = \limsup_{n \to \infty} \frac {n \log n}  {\log \frac{1}{\lambda_n|c_n|}} =  \limsup_{n \to \infty} \left(\frac {\log \frac{n!}{\lambda_n}}{n \log n}+ \frac {\log \frac{1}{|c_n|}}{n \log n}-\frac {\log n!}{n \log n}\right)^{-1} = \frac 1 {\omega^{-1}-1/2}. 
$$
Hence if $\omega<1$,  then $\rho<2$, which verifies \eqref{lac} in this case (since $s=0$).  

Suppose $\omega = 1$. Then $\rho =2$ and 
by \eqref{ordertype},
$$(\sigma e 2)^{{1}/{2}} = \limsup_{n \to \infty}n^{{1}/{2}} (\lambda_n|c_n|)^{{1}/{n}} = \limsup_{n \to \infty}n^{{1}/{2}} \bigg(\frac{\lambda_n}{n!}\bigg)^{{1}/{n}}(n!)^{{1}/{n}}|c_n|^{{1}/{n}}.$$
Since $(n!)^{{1}/{n}} \sim ne^{-1}$,   
$$(\sigma e 2)^{{1}/{2}} = e^{-1}\limsup_{n \to \infty} n^{{1}/{2}} |c_n|^{{1}/{n}}n\bigg(\frac{\lambda_n}{n!}\bigg)^{{1}/{n}} = e^{-1}(ce2)^{{1}/{2}}\tau e,$$
that is,  $\sigma = c\tau^2$. Now 
$$
\tau = s+  \limsup_{r \to \infty} \sum_{k=1}^\infty \frac { \log \left(1+ r/ {x_k} \right)} {r/x_k}x_k^{-1}=s, 
$$
since $\sum_{k=0}^\infty x_k^{-1} <\infty$, $r^{-1}\log (1+r)$ is bounded on $(0,\infty)$ and tends to zero as $r \to \infty$. Hence $\sigma=cs^2$ and the lemma follows. 
\end{proof}
We may now establish when $b \geq 1/4$ in \eqref{h2} and thus finish our proof of Theorem~\ref{piot}. 
Since the order and type with respect to $y$ of \eqref{h1} 
 is $2$ and ${1}/{4}$, it follows by Lemma \ref{ordertypelemma} that $b = s^2/4$. Theorem~\ref{piot} now follows from the following lemma of Craven and Csordas:
 
 \begin{lemma}[Lemma 2.2, \cite{CC2}]\label{CrCs}
Let $\{\lambda_n\}_{n=0}^\infty$ be a non-negative multiplier sequence with exponential generating function given by \eqref{psrep}. 
Then $s \geq 1$ if and only if $\lambda_0 \leq \lambda_1 \leq \lambda_2 \leq \cdots .$
\end{lemma}
\hfill \ensuremath{\Box}
\medskip 

\noindent
{\bf Acknowledgements.} We thank the two anonymous referees for several valuable suggestions that improved the exposition.

\end{document}